\documentclass[a4paper]{amsart}

\title{Tannaka--Kre\u{\i}n duality for Roelcke-precompact non-archimedean Polish groups}

\author{R\'emi Barritault}

\address{
  Universite Claude Bernard Lyon 1, ICJ UMR5208, CNRS, Ecole Centrale de Lyon, INSA Lyon, Université Jean Monnet, 69622 Villeurbanne, France.}
\email{barritault@math.univ-lyon1.fr}

\usepackage{a4wide}
\usepackage[T1]{fontenc}
\usepackage{kpfonts}

\usepackage{tikz}
\usetikzlibrary{cd}
\usepackage{graphicx}

\usepackage{amsmath, amssymb, amsthm, bbm, enumitem}
\usepackage[hidelinks]{hyperref}

\newcommand{\N}{\mathbb{N}}

\newcommand{\Q}{\mathbb{Q}}

\newcommand{\C}{\mathbb{C}}
\newcommand{\Sp}{\mathbb{S}}
\newcommand{\U}{\mathbf{U}}

\newcommand{\1}{\mathbbm{1}}

\DeclareMathOperator{\Aut}{Aut}

\DeclareMathOperator{\Homeo}{Homeo}
\DeclareMathOperator{\Hom}{Hom}
\DeclareMathOperator{\GL}{GL}
\DeclareMathOperator{\Sym}{Sym}

\renewcommand{\to}{\longrightarrow}
\renewcommand{\mapsto}{\longmapsto}

\DeclareMathOperator{\dom}{dom}
\DeclareMathOperator{\acl}{acl}

\DeclareMathOperator{\Ind}{Ind}

\newcommand{\M}{\mathcal{M}}

\newcommand{\B}{\mathcal{B}}
\newcommand{\Li}{\mathcal{L}}
\newcommand{\Hi}{\mathcal{H}}
\newcommand{\Ki}{\mathcal{K}}
\newcommand{\Un}{\mathcal{U}}

\newcommand{\A}{\mathcal{A}_G}
\newcommand{\spa}{\spec_{\A}}
\newcommand{\spec}{\mathbf{S}}
\newcommand{\triv}{\mathbf{1}}
\DeclareMathOperator{\FP}{F}
\DeclareMathOperator{\Hilb}{Hilb}

\newcommand{\Rep}{\mathbf{Rep}}
\newcommand{\rpna}{\mathbf{RPnA}}
\newcommand{\CR}{\Rep(\rpna)}
\newcommand{\Nat}{\mathcal{N}\hspace{-0,5mm}at}

\newcommand{\Hc}{\mathbf{H}}
\newcommand{\Pa}{\mathbf{P}}
\newcommand{\T}{\mathbf{T}}

\theoremstyle {plain}          \newtheorem{theo}{Theorem}[section]
\theoremstyle {plain}          \newtheorem{lem}[theo]{Lemma}
\theoremstyle {plain}          \newtheorem{coro}[theo]{Corollary}
\theoremstyle {plain}          \newtheorem{prop}[theo]{Proposition}
\theoremstyle {definition}     \newtheorem{defi}[theo]{Definition}
\theoremstyle {definition}     
\theoremstyle {definition}     \newtheorem{rem}[theo]{Remark}
\theoremstyle {definition}     

\theoremstyle {plain}          \newtheorem{fact}{Fact}
\theoremstyle {plain}          \newtheorem{theo'}{Theorem}

\begin{document}

\maketitle

\begin{abstract}
    Let $G$ be a Roelcke-precompact non-archimedean Polish group, $\B(G)$ the algebra of matrix coefficients of $G$ arising from its continuous unitary representations. The Gel'fand spectrum $\Hc(G)$ of the norm closure of $\B(G)$ is known as the \textit{Hilbert compactification} of $G$. Let $\A$ be the dense subalgebra of $\B(G)$ generated by indicator maps of open cosets in $G$. We prove that multiplicative linear functionals on $\A$ are automatically continuous, generalizing a result of Kre\u{\i}n for finite dimensional representations of topological groups. We deduce two abstract realizations of $\Hc(G)$. One is the space $\Pa(\M_G)$ of partial isomorphisms with algebraically closed domain of $\M_G$, the countable set of open cosets of $G$ seen as a homogeneous first order logical structure. The other is $\T(G)$ \textit{the Tannaka monoid} of $G$. We also obtain that the natural functor that sends $G$ to the category of its representations is full and faithful.
\end{abstract}

\tableofcontents

\section*{Introduction}

A fundamental question in abstract harmonic analysis is: How much information about a topological group can one recover from its representation theory ? In this paper, we are interested in cases where the group can be fully reconstructed.

More precisely, a \textit{unitary representation} of a topological group $G$ is a continuous group morphism $\pi$ from $G$ to the unitary group $\Un(\Hi_\pi)$ of a complex Hilbert space $\Hi_\pi$. Continuity means here that all the maps of the form $g \longmapsto \left<\pi(g)\xi,\eta\right>$ for $\xi, \eta \in \Hi_\pi$, called the \textit{matrix coefficients} of $\pi$, are continuous. $\pi$ is said \textit{irreducible} if $\Hi_\pi$ admits no non-trivial $G$-invariant closed subspace. In the well-behaved cases, every unitary representation of $G$ decomposes uniquely into an aggregate of irreducible subrepresentations. Harmonic analysis on such a group $G$ reduces to the study of its \textit{unitary dual} $\widehat{G}$, namely the collection of equivalence classes of irreducible unitary representations of $G$. In various subcases, duality theories have been established, allowing for the abstract reconstruction of $G$ from $\widehat{G}$. 

A fundamental instance of this is the case of locally compact abelian groups with the celebrated Pontryagin--van Kampen duality theory \cite{Pontrjagin_1934,Kampen_1935}. Indeed, if $G$ is locally compact and abelian, its irreducible representations always have dimension $1$ and $\widehat{G}$ can be identified with $\Hom(G, \Sp^1)$, the continuous group morphisms from $G$ to the unit circle. In particular, $\widehat{G}$ is also a group that is moreover locally compact and abelian. Finally, $G$ is canonically isomorphic to its bi-dual. 

Another context where representation theory is very tame is the compact case. Indeed, recall the Peter-Weyl Theorem which states, in particular, that every representation of a compact group splits in an essentially unique way as a sum of finite dimensional irreducible subrepresentations. This ideal situation allowed Tannaka and Kre\u{\i}n to develop, independently and with different approaches, duality theories for compact groups.

Kre\u{\i}n considered the algebra $\B'(G)$ generated by the matrix coefficients of a general topological group $G$ arising from its finite dimensional representations and proved the following technical but crucial result: positive linear functionals on $B'(G)$ are automatically continuous \cite{Krein_1941}. As a consequence, the \textit{Gel'fand spectrum} of $\B'(G)$ and that of its completion coincide. This is a compact topological group known as the \textit{Bohr compactification} $bG$ of $G$ and if $G$ is compact, then $G$ and $bG$ are canonically isomorphic.

Tannaka obtained similar duality results from a different perspective. He associated to a compact group $G$ a monoid $\T(G)$ of \textit{operations} on the class of representations of $G$. More explicitly, an element of $\T(G)$ is a family of operators $(u_\pi)_\pi$ where $\pi$ ranges over all the finite dimensional representations of $G$ and $u_\pi$ is an operator on the same the Hilbert space as $\pi$. Moreover, the family must commute with representation morphisms and preserve the common operations on representations: sum, tensor product and conjugation. The monoid law is pointwise composition. This structure can be endowed with a natural topology and is in a fact a compact group canonically isomorphic to $G$ \cite{Tannaka}.

The wider class of \textit{Roelcke-precompact} groups seems to retain a lot of the geometrical properties of compact groups. A topological group $G$ is \textit{Roelcke-precompact} if it is precompact in the Roelcke uniformity, in other words if for every open neighborhood $U$ of the identity, there exists a finite subset $F$ of $G$ such that $G = UFU$. Roelcke-precompact Polish groups are receiving an increasing amount of interest. Cameron started the investigation by carrying an extensive study of the dynamical properties of Roelcke-precompact permutation groups \cite{Cameron_1990}. Uspenksij \cite{Uspenskij_1998, Uspenskij_2001} and Glasner \cite{Glasner_2012} showed that well known groups, such as the unitary group of the separable Hilbert space or $\Aut(\mu)$ for an atomless standard Borel probability measure, are Roelcke-precompact and deduced strong properties such as minimality. \cite{Ibarlucia_2016} and the pair of papers \cite{BenYaacov_2016, BenYaacov_2018} studied compactifications of such groups. More recently, Ibarlucìa showed that Roelcke-precompact Polish groups have Kazdhan's Property (T) \cite{Ibarlucia_2021}, like compact group. Moreover, the Roelcke-precompact Polish groups that are \textit{non-archimedean}, i.e. that admit a basis of identity neighborhoods consisting of open subgroups, have seen their unitary representations fully classified \cite{Tsankov_2012} in a way that reminds a lot of the Peter-Weyl Theorem. Indeed, let $\M_G$ be the set of left translates of open subgroups of $G$. Then $G$ acts continuously on $\M_G$ seen as a discrete space and it gives rise to a unitary representation: $$\Lambda_G \colon G\curvearrowright \ell^2(\M_G).$$ 
This canonical construction captures all the representation theory of $G$, reminding of the left-regular representation of a compact group. More precisely, the following holds:

\begin{fact}\label{oligorep}
Let $G$ be a Roelcke-precompact non-archimedean Polish group.
\begin{enumerate}
    \item Every unitary representation of $G$ splits a sum of irreducible subrepresentations.
    \item Every irreducible unitary representation of $G$ is isomorphic to a subrepresentation of the representation $\Lambda_G$, which is faithful.
\end{enumerate}
\end{fact}

The first item is contained in \cite[Th. 4.2]{Tsankov_2012}. The second item is a consequence of the same result and is proved as part of Lemma \ref{tensor} below.

The main source of examples for such groups is model theory. Indeed, by the classical Ryll--Nardzewski Theorem (see e.g. \cite[Th. 7.3.1]{Hodges_1993}) the automorphism group of any $\aleph_0$-categorical model theoretical structure is part of this class. This includes the group $S_\infty$ of all permutations of the countable set, the group $\Aut(\Q,<)$ of order preserving bijections of $\Q$, $\Homeo(2^\N)$, $\Aut(R)$ the automorphism group of the random graph, or the group $\GL(\infty, q)$ of linear automorphisms of the countably infinite vector space over the finite field $\mathbb{F}_q$.

There are however fundamental properties that are lost in the Roelcke-precompact case: their irreducible representations can have infinite dimension and there is no Haar measure available. Using dynamical and model-theoretical properties of Roelcke-precompact non-archimedean Polish groups, we are still able to carry out similar constructions as Tannaka's and Kre\u{\i}n's. In particular, we obtain an analogue of Kre\u{\i}n's technical result. Indeed, let $\B(G)$ denote the set of matrix coefficients of $G$ arising from all the unitary representation of $G$ and let $C_b(G)$ denote the algebra of bounded continuous maps $G \longrightarrow \C$. Considering sums, tensor products and conjugates of representations, it is easily seen that $\B(G)$ is a subalgebra of $C_b(G)$, closed under complex conjugation. Let $\A$ be the linear span of the indicator maps of open cosets in $G$. This is in turn a subalgebra of $\B(G)$ (see Remark \ref{coset} Item 3) which is dense by Fact \ref{oligorep}. We obtained the following:

\begin{theo'}\label{Krein_Intro}
    Let $G$ be Roelcke-precompact non-archimedean Polish group. Multiplicative linear functionals $\tau \colon \A \longrightarrow \C$ are automatically positive and continuous.
\end{theo'}

Again, it implies that the Gel'fand spectra of $\A$ and its completion (i.e. the completion of $\B(G)$) coincide. Since this spectrum is compact, it cannot be isomorphic to $G$ in full generality. However, it can be endowed with a structure of \textit{semi-topological *-monoid}. This structure is known as the \textit{Hilbert compactification} $\Hc(G)$ of $G$ and has received quite some attention \cite{Glasner_2014, BenYaacov_2016}. Along the way, we obtain an identification of the Hilbert compactification of $G$ as the monoid of partial isomorphisms with algebraically closed domains of $\M_G$ seen as a homogeneous model theoretical structure (complete definitions are given in section 1).

Tannaka's approach also bears fruit in this context, providing another realization of $\Hc(G)$. Indeed, consider the category $\Rep(G)$ of unitary representations of $G$ that are isomorphic to a finite sum of irreducible representations. Let $\T(G)$ be the monoid of families of operators $(u_\pi)_{\pi \in R}$ that commute with sums, tensor products and representation morphisms (complete definitions are given in section \ref{Tannaka} Definition \ref{Tmonoid}). We obtained the following:

\begin{theo'}\label{Tannaka_Intro}
    Let $G$ be a Roelcke-precompact non-archimedean Polish group. Then $T(G)$ is a compact semi-topological *-monoid canonically isomorphic to the Hilbert compactification of $G$.
\end{theo'}

Finally, we establish that the Hilbert compactification fully remembers the original group. Indeed, a Roelcke-precompact non-archimedean Polish group $G$ is homeomorphic to the set of invertible elements of $\Hc(G)$ via the canonical map that sends $g$ in $G$ to the \textit{evaluation map at} $g$. Moreover, we can form the category whose objects are the $\Rep(G)$ for every Roelcke-precompact non-archimedean Polish group $G$. With the right arrows (the \textit{admissible} functors, see section \ref{Conc} Definition \ref{admissible}), we have the following:

\begin{theo'}\label{recons_Intro}
\;

    \begin{enumerate}
        \item Let $G$ be a Roelcke-precompact non-archimedean Polish group. The canonical map $G \longrightarrow \Hc(G)$ is a homeomorphic embedding that respects the algebraic structures. Its image is the set of invertibles of $\Hc(G)$.

        \item The contravariant functor $\Rep$ that sends a Roelcke-precompact non-archimedean Polish group to the category of its unitary representations is full and faithful, i.e. a \textit{duality}.
    \end{enumerate}

\end{theo'}

Numerous treatments, variants and generalizations of the Tannaka--Kre\u{\i}n duality can be found in the literature. An extensive exposition of the theory, recounting both Kre\u{\i}n's and Tannaka's approaches successively, appears in \cite{Hewitt_1970}. A more modern take on Tannaka's point of view, phrased in the language of category theory, which also lists developments and applications, figures in \cite{Joyal_1991}. Tannaka--Kre\u{\i}n duality has sprouted roots in different domains of study. One is category theory, with the works of Grothendieck, Deligne and Saavedra-Rivano among others, who introduced \textit{Tannakian formalism} to reverse the process of the classical duality. The idea is that to a \textit{Tannakian category} $\mathcal{C}$ can be associated an algebraic group $G$ such that $\mathcal{C}$ is equivalent to $\Rep(G)$ (see for instance \cite{SaavedraRivano_1972}). More recently, Tannaka--Kre\u{\i}n duality theory has been reappropriated in a prolific way for the study of \textit{quantum groups} and \textit{knot theory}. It also appears in mathematical physics, related to the \textit{superselection principle} in quantum field theory. See the survey \cite{Vainerman_2015}.

This article is organized as follows. In the first section we introduce the notations and some necessary preliminary results. The second section is inspired from Kre\u{\i}n's approach to the duality and contains a proof of Theorem ~\ref{Krein_Intro}. In section \ref{Tannaka}, where the formalism follows that of \cite{Joyal_1991}, we switch to Tannaka's point of view and use the automatic continuity property to prove Theorem \ref{Tannaka_Intro}. In the last section, we prove Theorem \ref{recons_Intro}.

\subsection*{Acknowledgments}

I am deeply grateful to Todor Tsankov for suggesting me to investigate on this topic and for his support in the making of this article.

\section{Preliminaries}

In this section, we set the global framework by fixing notations and stating useful general facts. We start with dynamical and model-theoretical notions. 

Let $G$ be a \textit{Polish group} i.e. a separable and completely metrizable topological group. For subgroups $U,V$ of $G$, we will write 
        $$U\backslash G/V = \left\{ UgV,\ g \in G\right\}.$$
This is a partition of $G$. Note that if $G$ is non-archimedean, then it is Roelcke-precompact if and only if $U\backslash G/U$ is finite for every open subgroup $U$ of $G$.

We will need the dynamical notion of algebraicity, which coincides with the model-theoretical one in the $\aleph_0$-categorical case but not in general. It is only assumed that the definition of first order structure, substructure and isomorphism, in the sense of \cite{Hodges_1993}, are known.

\begin{defi}
    Let $G$ be the automorphism group of some first-order structure $\M$ and $A \subseteq \M$.
\begin{itemize}
    \item[1.] We denote by $G_A$ the pointwise stabilizer of $A$ and by $G_{(A)}$ the setwise stabilizer of $A$.
    
    \item[2.] If $A$ is finite, an element $a \in \M$ is said to be \textit{algebraic over $A$} if the orbit $G_A\cdot a$ is finite. We denote by $\acl(A)$ the \textit{algebraic closure of $A$} i.e. the set of all elements of $\M$ that are algebraic over $A$. If $A$ is infinite, the algebraic closure of $A$ is $\acl(A) = \bigcup \{\acl(B),\ B\subseteq A \text{ finite}\}$.
    
    \item[3.] $A$ is said to be \textit{algebraically closed} if $\acl(A) = A$.
\end{itemize}
\end{defi}

Recall that $\acl$ is a closure operator, in particular it is non-decreasing and satisfies $\acl^2 = \acl$.

\begin{defi}
    Let $\M$ be a first-order structure. A \textit{partial automorphism} of $\M$ is an isomorphism between substructures of $\M$.
    
    We will see these maps as subsets of $\M\times \M$. Thus, for partial automorphisms $s,s'$ of $\M$, $s\subseteq s'$ means $s'$ extends $s$ and $s\cup s'$ denotes the unique common extension of $s$ and $s'$ to the substructure of $\M$ generated by $\dom(s)\cup\dom(s')$, if it exists.
    
    We will denote by $\Pa(\M)$ the set of partial automorphisms of $\M$ with algebraically closed domain and by F$(\M)$ the set of partial automorphisms of $\M$ with finite domain. Note that those sets are stable under \textit{composition where defined} and inversion. 
    
    Finally, $\M$ is \textit{homogeneous} if every element of $\FP(\M)$ extends to a \textit{total} automorphism, i.e. an element of $\Aut(\M)$.

\end{defi}

\begin{rem}\label{topo}
    Let $\M$ be a structure. $\Pa(\M)$ is a \textit{compact semi-topological *-monoid} i.e.:
    \begin{itemize}
        \item[--] it is stable under \textit{composition where defined} with a neutral element: $\operatorname{id_\M}$, 
        \item[--] it has an involutive anti-automorphism: $u \longmapsto u^* = \{(y,x), \ (x,y) \in u\}$,
        \item[--] it can be endowed with a compact topology that makes composition separately continuous and $^*$ continuous.
    \end{itemize}
    
    To define the topology, consider $K=\M\cup\{\infty\}$ the one-point compactification of the discrete structure $\M$. Elements $p\in K^K$ can be seen as partial maps $\M \longrightarrow \M$, undefined wherever $p(x) = ~\infty$. The domain of $p$ corresponds to the set $p^{-1}(\M)$. 
    
    $K^K$ is endowed with the product topology, which is compact and Hausdorff. In this space, $\Pa(\M)$ is exactly the closure of $G=\Aut(\M)$ hence compact. More details can be found in \cite[Prop. 3.8]{BenYaacov_2018}. Subbasic neighborhoods of an element $u_0 \in \Pa(\M)$ are of two kinds:
    $$O_{x} = \{u \in \Pa(\M), \  u(x) = u_0(x)\},$$
    where $x\in \dom(u_0)$ and
    $$U_{x,A} = \{u\in \Pa(\M),\ \left[x\in \dom(u) \Rightarrow u(x) \notin A\right]\},$$
    where $x \notin \dom(u_0)$ and $A$ is a finite subset of $\M$.

    Note also that $G$ naturally acts on $\Pa(\M)$ by left-composition which is hence a semi-topological *-monoidal \textit{compactification} of $G$. An extensive survey regarding such objects can be found in \cite{Glasner_2014}. 
\end{rem}

We will be needing the following combinatorial result, a variant of B.H. Neumann's Lemma appearing as \cite[Th.1]{Birch_1976}:

\begin{theo}[B.H. Neumann's Lemma]
    Let $G$ be a group acting on a set $X$ such that all orbits are infinite. For every finite subsets $A,B \subseteq X$, there exist $g \in G$ such that $g(A)\cup B = \emptyset$.
\end{theo}

More precisely, the following translation of the above Theorem will be key in establishing the main result of section \ref{Krein}.

\begin{lem}\label{keylem}
    Let $G$ the automorphism group of some homogeneous countable structure $\M$ and assume it is Roelcke-precompact. Let $s,s_1,...,s_n$ be partial automorphisms of $\M$ with finite domain such that:
    $$\forall i \leqslant n,\ \dom(s_i) \not \subseteq \acl(\dom(s))$$

    Then, there exist $g \in G$ that extends $s$ but none of the $s_i$'s.
\end{lem}

\begin{proof}
    By homogeneity of $\M$, we can pick $g_0 \in G$ extending $s$. Fix also $x_i \in \dom(s_i) \backslash \acl(\dom(s))$ for every $i\leqslant n$. Using Roelcke-precompactness of $G$, Lemma 2.4 in \cite{Evans_2016} tells us that $G_{\acl(\dom(s))} \curvearrowright \M \backslash \acl(\dom(s))$ only has infinite orbits. Now B.H. Neumann's Lemma above gives $u \in G_{\acl(\dom(s))}$ such that:
    $$u(\{x_1,...,x_n\}) \cap \{g_0^{-1}(s_1(x_1)),...,g_0^{-1}(s_n(x_n))\} = \emptyset.$$
    Then $g=g_0u$ extends $s$ but none of the $s_i$'s
\end{proof}

\begin{rem}\label{coset}
Let $G$ be a Roelcke-precompact non-archimedean Polish group. We will call (open) \textit{cosets} of $G$ the left-translates of (open) subgroups of $G$. The set of open cosets of $G$ will be denoted $\M_G$. We recall well-known facts about these objects and give sketches of proofs.

    \begin{enumerate}
        \item $\M_G$ is countable. Indeed, let $(U_n)$ be a countable basis of open neighborhood of $G$ consisting of open subgroups and let $D \subseteq G$ be a countable dense subset. Let $U$ be an arbitrary open subgroup of $G$. There exists $n\in \N$ such that $U_n \subseteq U$. Since $G$ is Roelcke-precompact, there is $F\subseteq D$ finite such that $U=U_nFU_n$. This only allows for countably many distinct open subgroups. Moreover, since every open coset $gU$ of $G$ must contain an element of $D$ by density, $G$ only has countably many open cosets.

        \item $G$ naturally acts on the left of $\M_G$. This action is continuous and the associated morphism $G \longrightarrow \Sym(\M_G)$ is a homeomorphic embedding. Because $G$ and $\Sym(\M_G)$ are both Polish, $G$ must be closed in $\Sym(\M_G)$. Thus, after naming the orbits of the diagonal actions $G \curvearrowright \M_G^n$ for every $n\in \N$, we can view $\M_G$ as a discrete first-order homogeneous structure such that:
            $$G = \Aut(\M_G).$$     

        \item Let $gU, hV$ be cosets of $G$. Then, $gU\cap hV \neq \emptyset$ if and only if there is $f \in G$ such that $fU = gU$ and $fV=hV$. In particular, if $gU\cap hV$ is non-empty, then it is a left-translate of $U\cap V$.
        
        \item Suppose $G= \Aut(\M)$ for some countable and homogeneous structure $\M$. There is a natural map $\FP(\M) \longrightarrow \M_G$. Indeed, for $s \in \FP(\M)$, fix an extension $g_0 \in G$. Then, a given automorphism $g$ of $\M$ extends $s$ if and only if $g \in g_0G_{\dom(s)}$ and this coset does not depend on the choice of $g_0$. Similarly, if $\dom(s) = A \subseteq B$ where $A$ and $B$ are finite, there is a bijection between $G_A/ G_B$ and $\{s'\in \FP(\M), \ s\subseteq s' \text{ and } \dom(s') = B\}$, namely $s' \longmapsto g_0^{-1}gG_B$ where $g$ is any extension of $s'$ (this map does depend on the choice of $g_0$).
    \end{enumerate}
\end{rem}

Switching to representation theory, we will write $\Hilb(G)$ for the closure of $\B(G)$ in $C_b(G)$ with respect to the norm topology. This is called the Hilbert algebra of $G$. The \textit{Hilbert compactification} $\Hc(G)$ of $G$ is the Gel'fand spectrum of $\Hilb(G)$, i.e. the weak*-compact space of non-zero multiplicative linear functionals on $\Hilb(G)$ that commute with complex conjugation. 

Given two representations $\pi, \pi'$ of $G$, a morphism of representations between $\pi$ and $\pi'$, or \textit{intertwining operator}, is a bounded operator $h\colon \Hi_\pi \to \Hi_{\pi'}$ that commutes with the action of $G$:
$$\forall g \in G, \ h\circ \pi(g) = \pi'(g) \circ h.$$
The set of such operators will be denoted either by $\Hom(\pi,\pi')$ or $\Hom_G(\Hi_\pi, \Hi_{\pi'})$.

The following are consequences of the results in \cite{Tsankov_2012}. For the definition and properties of \textit{induced representations}, we refer the reader to section 6 of \cite{Folland_2016}. It is usually defined in the locally compact setting using invariant measures but the basic results still hold in this context with the same proofs, and the formalism is even simpler. Indeed, our quotient spaces are all discrete hence can be endowed with the counting measure. The distinction between isomorphic representations is often omitted.

\begin{lem}\label{tensor}
    Let $G$ be a Roelcke-precompact non-archimedean Polish group
    \begin{enumerate}
        \item Every irreducible representation of $G$ is a subrepresentation of $\Lambda_G \colon \ell^2(\M_G)$ which is faithful.
        
        \item Let $\pi_1,\pi_2$ be two irreducible representations of $G$. Then $\pi_1\otimes \pi_2$ decomposes as a finite sum of irreducible subrepresentations.
    \end{enumerate}
\end{lem}

\begin{proof}

We first show that any irreducible representation of $G$ is a subrepresentation of a left-quasi-regular representation $\lambda_{G/V} \colon G \curvearrowright \ell^2(G/V)$ of $G$ for some open subgroup $V$ (the quotient space $G/V$ is discrete and endowed with the counting measure).

Let $\pi$ be an irreducible representation of $G$. Using Theorem 4.2 from \cite{Tsankov_2012}, $\pi$ is isomorphic to an induced representation of the form $\Ind_{H}^G(\sigma)$, where $H$ is an open subgroup of $G$ and $\sigma$ is an irreducible representation of $H$ that factors through a finite quotient of $H$. More explicitly, there exists $V$ an open normal subgroup of $H$ of finite index such that $\sigma$ is the pullback of an irreducible representation of the finite group $H / V$. 

It is a fact that the quasi-regular representations $\lambda_{G/ V} \colon G \curvearrowright \ell^2(G/ V)$ and $\lambda_{H/ V} \colon H\curvearrowright~\ell^2(H/ V)$ are linked in the following way:
\begin{equation}\label{induction}
    \lambda_{G/ V} = \Ind_V^G(\triv_V) = \Ind_H^G(\Ind_V^H(\triv_V)) = \Ind_H^G(\lambda_{H/V}),
\end{equation}
where $\triv_V$ denotes the trivial one-dimensional representation of $V$. The first and last equalities are basic properties of induction and the second one is given by the Theorem on induction by stages (see e.g. \cite[Theorem 6.14]{Folland_2016}). Now, applying the Peter-Weyl Theorem to the finite group $H/V$, we get that $\sigma$ is a subrepresentation of $\lambda_{H/V}$. Since induction preserves subrepresentations, we deduce from (\ref{induction}) that $\pi = \Ind_H^G(\sigma)$ is a subrepresentation of $\lambda_{G/ V}=\Ind_H^G(\lambda_{H/V})$. Since $\lambda_{G/V}$ is a subrepresentation of $\Lambda_G$, so is $\pi$. The faithfulness of $\Lambda_G$ is then a direct consequence of the fact that closed permutation groups satisfy the Gel'fand-Raikov Theorem \cite[Th. 1.1]{Tsankov_2012} hence Item 1 is proved.

Next, fix $\pi_1$ and $\pi_2$ two irreducible representations of $G$. From what precedes, we can find two open subgroups $V_1$ and $V_2$ of $G$ such that $\pi_i$ is a subrepresentation of $\lambda_{G/V_i}, \ $ $i=1,2$. Then:
    \begin{align*}
        \ell^2(G/ V_1) \otimes \ell^2(G/ V_2) &\simeq \ell^2(G/ V_1 \times G/ V_2)\\
        &\simeq \bigoplus_{f\in V_1\backslash G/V_2} \ell^2(G\cdot (V_1,fV_2))\\
        &\simeq \bigoplus_{f\in V_1\backslash G/V_2} \ell^2\left(G/ \left(V_1\cap fV_2f^{-1}\right)\right).
    \end{align*}

Since there is a natural surjective map $\; \left(V_1 \cap V_2\right) \backslash G  / \left( V_1 \cap V_2\right) \longrightarrow V_1\backslash G /V_2\;$ and $G$ is Roelcke-precompact, $V_1 \backslash G  / V_2$ is finite. In particular, we have shown that $\pi_1 \otimes \pi_2$ is a subrepresentation of a finite sum of quasi-regular representations.

To finish the proof, it suffices to show that every quasi-regular representation $\lambda_{G/V}$ for $V$ an open subgroup of $G$ splits as a finite sum of irreducibles. Indeed, let $H$ be the \textit{commensurator} of $V$, meaning:
$$H =\left\{g \in G, \ [V,V \cap gVg^{-1}] < \infty \quad \& \quad [gVg^{-1},V \cap gVg^{-1}]\ \right\}.$$

The hypotheses on $G$ implie that $V$ has finite index in $H$ \cite[Lem. 2.7]{Tsankov_2012}. Hence, 
$$V' := \bigcap_{h\in H} hVh^{-1},$$
is in fact a finite intersection of conjugates of $V$. Thus, it is open and has finite index in $V$ by the definition of $H$. 

To sum up, we have obtained that $V'$ is an open normal subgroup of $H$ with finite index. In particular, $H/V'$ is a finite group hence the quasi-regular representation $\lambda_{H/V'}$ of $H$ is finite dimensional. It thus splits as a finite sum of irreducible subrepresentations:
$$\lambda_{H/V'} = \bigoplus_{1\leqslant i \leqslant n} \sigma_i.$$
Then, by the basic properties of induction:
$$ \lambda_{G/ V'} =\Ind_H^G(\lambda_{H/V}) = \bigoplus_{1\leqslant i \leqslant n} \Ind_H^G(\sigma_i)$$
Finally, each of the $\Ind_H^G(\sigma_i)$ is irreducible by \cite[Prop. 4.1]{Tsankov_2012}. Since $V' \subseteq V,\ \lambda_{G/V}$ is a subrepresentation of $\lambda_{G/V'}$ hence also splits as a finite sum of irreducibles.

\end{proof}

In the next sections, we give two different characterizations of $\Hc(G)$ for a Roelcke-precompact non-archimedean Polish group $G$. For a detailed study of the Hilbert compactification of such groups, see \cite{BenYaacov_2018}.

\section{Application to Kre\u{\i}n's duality}\label{Krein}

In \cite{Krein_1941}, Kre\u{\i}n associated to a compact group $G$ what is now called a \textit{Kre\u{\i}n algebra}. It is built on the algebra of finite dimensional matrix coefficients of $G$ together with a specific basis \textit{and} the information about how tensor products of irreducible representations decompose. We describe a similar construction: we define a canonical dense subalgebra $\A$ of $\B(G)$ for $G$ a Roelcke-precompact non-archimedean Polish group, fixed from now on. In some cases, it comes with a canonical basis. However, at the intersection of both contexts, while the algebra we get coincides with the one appearing in Kre\u{\i}n's work, the basis we (sometimes) obtain is different. 

The main result of this section is the automatic continuity of multiplicative linear functionals on that canonical algebra, an analogue of a central technical lemma in the establishment of the original duality. A direct consequence of this is the identification in a canonical way of the Gel'fand spectra of three algebras: $\A,\ \B(G)$ and $\Hilb(G)$, the spectrum of the latter being the definition of $\Hc(G)$. We also obtain a model theoretic description of $\Hc(G)$ as $\Pa(\M_G)$. 

Recalling Fact \ref{oligorep}, we consider the 'universal' representation $\Lambda_G \colon G \curvearrowright \ell^2(\M_G)$ and form the matrix coefficients obtained from the canonical basis $(\delta_x)_{x \in \M_G}$ of $\ell^2(\M_G)$. These are all the maps of the form 
$$f_{x,y} \colon G \to \C, \qquad  g \mapsto \left<g\cdot \delta_x, \delta_y\right> = \left\{
    \begin{array}{ll}
        1 & \mbox{if } g(x) = y, \\
        0 & \mbox{otherwise.}
    \end{array}
\right.$$ 
for $x,y \in \M_G$. We will denote by $\A$ the linear span of these maps. The set of these generators is exactly the set of indicator maps of open cosets in $G$ together with the zero function. In particular, $\A$ contains the constant functions. Moreover, recalling Item 3 in Remark \ref{coset}, $\A$ is in fact a unitary subalgebra of $\B(G)$ which moreover is norm-dense by Fact \ref{oligorep}. It is also stable under the involution $f\longmapsto  \bar{f}$ of $\B(G)$ induced by complex conjugation as well as the one induced by inversion in $G$.

Alternatively, $\A$ can be described in terms of finite partial automorphisms of $\M_G$. Indeed, define for every $s\in \FP(\M_G)$ a map 
$$e_s \colon G \longrightarrow \C, \qquad g \longmapsto \left\{
    \begin{array}{ll}
        1 & \mbox{if } g \mbox{ extends }s, \\
        0 & \mbox{otherwise.}
    \end{array}
\right.$$
It follows from Item 4 in Remark \ref{coset} that these are also exactly the indicator maps of open cosets in $G$. Note that for $s,s'\in \FP(\M_G)$, we have:
\begin{equation}\label{mult}
    e_s\cdot e_{s'} = \left\{
    \begin{array}{ll}
        e_{s\cup s'} & \mbox{if } s\cup s' \in \FP(\M_G), \\
        0 & \mbox{otherwise.}
    \end{array}
\right.
\end{equation}

A special case of interest is when $G$ can be presented as the automorphism group of some $\aleph_0$-categorical structure $\M$ that admits \textit{weak elimination of imaginaries}. For the formal definition, we refer the reader to section 4.2 of \cite{Hodges_1993}. Informally, an $\aleph_0$-categorical structure $\M$ eliminates imaginaries if we can recover all open subgroups of $G = \Aut(\M)$ from the action $G \curvearrowright \M$ up to finite index. More precisely, if $\M$ is an $\aleph_0$-categorical structure, the following property can be taken as a definition of \textit{weakly eliminating imaginaries} (see \cite{Tsankov_2012} Lemma 5.1): For every open subgroup $U$ of $G$, there exists a unique algebraically closed finite substructure $A$ of $\M$ such that 
$$G_A \subseteq U \subseteq G_{(A)}.$$
Note that $G_{(A)}/G_A$ is isomorphic to $\Aut(A)$, hence finite. In particular, $G_A$ has finite index in $U$. It is a classical fact that all the examples from the introduction weakly eliminate imaginaries (see for example section 4.2 of \cite{Hodges_1993}).

In that case, we can extract a canonical basis from the previous generating family of $\A$. Namely, defining $e'_s \colon G \longrightarrow \{0,1\}$ as above for every $s\in \FP(\M)$, we have the following:

\begin{prop}
    Let $\M$ be a countable $ \aleph_0$-categorical structure that admits weak elimination of imaginaries and let $G = \Aut(\M)$. Then
    $$B = \big{(}e'_s, \ s\in \FP(\M),\ \dom(s) = \acl(\dom(s))\big{)}$$
    is a basis of $\A$ (and a subfamily of the previous generating family).
\end{prop}

\begin{proof}
    First, recall that in an $\aleph_0$-categorical structure, algebraic closures of finite sets are finite and finite partial automorphisms always extend to total automorphisms. Hence, using Item 4 in Remark \ref{coset}, the elements of $B$ are indeed indicator functions of open cosets, namely of all cosets associated to subgroups of the form $G_{\acl(A)}$ for $A\subseteq \M$ finite.
    
    Now if $U$ is an open subgroup of $G$, we can find by weak elimination of imaginaries a finite algebraically closed substructure $A\subset \M$ such that $V = G_{\acl(A)} \subseteq U$ with finite index. Then, writing $U = u_1V \sqcup ... \sqcup u_nV$ for some $u_1,...,u_n\in U$, we have for every $g \in G$:
    $$\1_{gU} = \1_{gu_1V} + ... + \1_{gu_nV},$$
    where $\1_X$ denotes the indicator function of the set $X \subseteq G$. In particular, $B$ generates $\A$.

    It remains to see that $B$ is free. Assume we have an equation of the form $\lambda_1 e'_{s_1} + ... + \lambda_n e'_{s_n} = 0$ for some $\lambda_1,...,\lambda_n \in \C$ and distinct $s_1,...,s_n\in \FP(\M)$ with algebraically closed domains. Let $i$ be such that $\dom(s_i)$ is minimal for inclusion. Using Lemma \ref{keylem}, there exists $g \in G$ that extends $s_i$ without extending $s_j$ for any $j\neq i$ such that $\dom(s_j) \nsubseteq \dom(s_i)$. Note that if $j\neq i$ and $\dom(s_j) \subseteq \dom(s_i)$ then the domains are equal by choice of $i$ and there exists $a\in \dom(s_i)$ such that $s_j(a) \neq~s_i(a) =~g(a)$. Thus $g$ only extends $s_i$. Evaluating the above linking equation in $g$ yields $\lambda_i = 0$. We conclude by induction on $n$.

\end{proof}

\vspace{0,3cm}

The main object of study in this section is the Gel'fand spectrum of $\A$, denoted $\spa$. In the following Lemma, we do not need to assume that $G$ is Roelcke-precompact nor Polish.

\begin{lem}\label{uphi}
    Let $G$ be a non-archimedean topological group. Let $\phi \colon \A \to \C$ be a non-zero multiplicative linear functional. There exists a unique partial automorphism $u_\phi$ of $\M_G$ with algebraically closed domain such that:
    \begin{equation}   
    \forall s \in \FP(\M_G),\ \phi(e_s) =  \left\{
    \begin{array}{ll}
        1 & \mbox{if } u_\phi \mbox{ extends }s, \\
        0 & \mbox{otherwise.}
    \end{array}
    \right.
    \end{equation}
\end{lem}

\begin{proof}
    First note that for every $s\in \FP(\M_G),\ \phi(e_s) \in \{0,1\}$ since $e_s^2 = e_s$. Then, let 
$$u_\phi =\bigcup_{\substack{s\in \FP(\M_G) \\ \phi(e_s) = 1}}s.$$
We will show that $u_\phi$ satisfies the claim.

        \begin{itemize}
            \item[--] \underline{$u_\phi$ is a function:}

    Let $x \in \M_G$ be such that there exists $s,s' \in \FP(\M_G)$ with $x\in \dom(s) \cap \dom(s')$ and $\phi(e_s) = 1 = \phi(e_{s'})$. By multiplicativity of $\phi$, we have 
    $$\phi(e_s e_{s'}) = \phi(e_s)\phi(e_{s'}) = 1.$$
    In particular, $e_s e_{s'} \neq 0$ which implies the existence of $g \in G$ that extends both $s$ and $s'$. Thus $s(x)=s'(x)$ and $u_\phi$ is a partial function $\M_G \longrightarrow \M_G$.

            \item[--] \underline{$u_\phi$ is a partial isomorphism:}
            
    Since $\M_G$ is a relational structure, every set in $\M_G$ is a substructure. Thus, it is enough to show that each restriction of $u_\phi$ to a finite set extends to an automorphism of $\M_G$. To that aim, let $A=\{a_1,...,a_n\}\subseteq \dom(u_\phi)$ be finite. By definition of $u_\phi$, there exists $s_1,...,s_n \in \FP(\M_G)$ such that for all $i\leqslant n$, $\ a_i \in \dom(s_i)$ and $\phi(e_{s_i}) = 1$. Up to replacing $s_i$ with $s_{i|A}$, we can assume $\dom(s_i) \subseteq A$ for every $i\leqslant n$. Using the multiplicativity of $\phi$, we get:
    $$\phi(e_{s_1} \cdots e_{s_n}) = \phi(e_{s_1}) \cdots \phi(e_{s_n}) = 1.$$
    In particular, $e_{s_1} \cdots e_{s_n} \neq 0$ and there exists $g \in G$ a common extension of $s_1,...,s_n$ hence of $u_{\phi|A}$. Note also that, as $e_{s_1} \cdots e_{s_n} = e_{s_1\cup \cdots \cup s_n} =  e_{u_{\phi|A}}$, we have that $\phi(e_{u_{\phi|A}}) = 1$.

            \item[--] \underline{$\dom(u_\phi)$ is algebraically closed:}

    Let $A\subseteq \dom(u_\phi)$ be finite and $b \in \acl(A)$.    
    As noted above, $\phi(e_{u_{\phi|A}}) = 1$ and there exists $g_0 \in G$ that extends $u_{\phi|A}$. An element $g \in G$ that coincides with $u_\phi$ on $A$ lies in $g_0G_A$ hence must send $b$ in the finite set $g_0G_A \cdot b$. In other terms, writing $g_0G_A\cdot b =\{g_0(b),...,g_n(b)\}$ for some $g_1,...,g_n \in G$, we have:
    $$e_{u_{\phi|A}} = e_{g_{0|A\cup\{b\}}} + ... + e_{g_{n|A\cup\{b\}}}.$$
    Applying $\phi$ to the above equation gives by linearity that:
    $$1 = \phi(e_{g_{0|A\cup\{b\}}}) + ... + \phi(e_{g_{n|A\cup\{b\}}}).$$ 
    Since $\phi$ is $\{0,1\}$ valued at idempotents, there exists $i\leqslant n$ such that $\phi(e_{g_{i|A\cup\{b\}}}) =~ 1$. This shows that $b$ lies in $\dom(u_\phi)$ which is thus algebraically closed.
    
        \end{itemize}

    Now that we have built $u_\phi$, recalling that $\phi(u_{\phi|A}) =1$ for every finite subset $A$ of $\M_G$, the Theorem is proved.
\end{proof}

In the Roelcke-precompact and Polish case, this construction yields an identification of the Gel'fand spectrum $\spa$ of $\A$:

\begin{coro}\label{spa}
    Let $G$ be a Roelcke-precompact non-archimedean Polish group. The map 
    \begin{align*}
    \spa &\longrightarrow \Pa(\M_G)\\
    \phi & \longmapsto u_\phi
    \end{align*}
    is a homeomorphism.
\end{coro}

\begin{proof}
    It is straightforward from the previous result that the map is injective.
    We now show surjectivity, i.e. that if $u\in \Pa(\M_G)$ is fixed, the conditions
    \begin{equation*}    
    \forall s \in \FP(\M_G),\ \phi(e_s) =  \left\{
    \begin{array}{ll}
        1 & \mbox{if } u\mbox{ extends }s, \\
        0 & \mbox{otherwise,}
    \end{array}
    \right.
    \end{equation*}
    always define a multiplicative linear functional $\phi_u \colon \A \longrightarrow \C$. To that aim, assume we have an equation of the form
    \begin{equation}\label{eq}
    \sum_{i\leqslant n} \lambda_i e_{s_i} = 0
    \end{equation}
    for some $s_1,...,s_n\in \FP(\M_G)$ and $\lambda_i\in\C$. Recalling the identification $G = \Aut(\M_G)$ from Remark \ref{coset}, we set up for the application of Lemma \ref{keylem}.

    We define the following sets:
    $$A = \bigcup\{\dom(s_i),\ i\leqslant n,\ \dom(s_i)\subseteq \dom(u)\} \ \text{ and } \ I = \{i \leqslant n, \ \dom(s_i) \nsubseteq \acl(A)\}.$$
    By Lemma \ref{keylem}, there exists $g_0 \in G$ that extends $u_{|A}$ but none of the $s_i$ for $i\in I$. We claim that it satisfies the following:
    \begin{equation}
    \forall i\leqslant n, \ \left[ s_i\subseteq g_0 \iff s_i\subseteq u \right].
    \end{equation}

    Indeed, assume that $s_i \nsubseteq u$. Then either $\dom(s_i)$ is included in $\dom(u)$ or not. In the first case, there must exist $a \in \dom(s_i)$ such that $s_i(a) \neq u(a) = g_0(a)$ hence $s_i \nsubseteq g_0$. In the second case, we also have $\dom(s_i) \nsubseteq \acl(A)$ otherwise $\dom(s_i) \subseteq \acl(\dom(u)) = \dom(u)$. Hence $i \in I$ and $s_i \nsubseteq g_0$. The converse is clear by construction.

    Now, evaluating (\ref{eq}) in $g_0$ yields:
    \begin{equation*}
    0 = \sum_{i\leqslant n}\lambda_ie_{s_i}(g_0) = \sum_{s_i\subseteq g_0} \lambda_i = \sum_{s_i\subseteq u} \lambda_i.    
    \end{equation*}
    Thus $\phi_u$ is well defined and linear. We check the multiplicativity of $\phi_u$ on the spanning subset $\{e_s,\ s \in \FP(\M_G)\}$. Recalling (\ref{mult}), it follows from the observation that, for every $s,s'\in ~\FP(\M_G)$, we have:
    $$\phi_u(e_se_{s'}) = 1 \ \iff \ \left[s\cup s' \in \FP(\M_G) \text{ and } s\cup s' \subseteq u \right]\ \iff \ s,s' \subseteq u \ \iff \phi_u(e_s) = 1 = \phi_u(e_{s'}).$$
    Finally, it is non-zero since the empty map $\emptyset$ belongs to $\FP(\M_G)$ and we always have $\emptyset \subseteq u$.

    To finish the proof, recall that $\Pa(\M_G)$ is compact and Hausdorff. Since $\spa$ is Hausdorff too, we can use Poincaré's Theorem and it suffices to show that $\Pa(\M_G) \longrightarrow \spa$ is continuous.

    Recall that $\{e_s, \ s \in \FP(\M_G)\}$ are idempotents and linearly span $\A$. In particular, $\phi(e_s) \in ~\{0,1\}$ for every $\phi\in \spa$ and $s \in \FP(\M_G)$ hence the following sets form a subbasis for the topology on $\spa$:
    $$O_{s,\varepsilon} = \left\{\phi\in \spa,\ \phi(e_s) = \varepsilon\right\},  \qquad \text{ for } s \in \FP(\M_G) \text{ and } \varepsilon \in \{0,1\}.$$

    Now, let $s \in \FP(\M_G)$ and $\varepsilon \in \{0,1\}$. We treat the case $\varepsilon = 0$, the other is similar. We have:
    \begin{align*}
        \phi \in O_{s,0} 
        &\iff \phi(e_s) = 0\\
        &\iff s \nsubseteq u_\phi \\
        &\iff  \exists x \in \dom(s), \ \left[x \notin \dom(u_\phi) \text{ or } u_\phi(x) \neq s(x) \right]\\
        &\iff u_\phi \in \bigcup_{x\in \dom(s)}\left\{v\in \Pa(\M_G), \ x \in \dom(v) \Rightarrow v(x) \notin\{s(x)\} \right\}.
    \end{align*}
    Recalling the definition of the topology on $\Pa(\M_G)$ from Remark \ref{topo}, we see that the desired continuity holds.
    
\end{proof}

We can now turn to the main result of this section. As mentioned in the introduction, the next Theorem is an adaptation of a technical but crucial Lemma from the compact duality: Kre\u{\i}n proved \cite{Krein_1941} that for a compact group $K$, positive multiplicative linear functionals on $\B'(K)$ are automatically continuous. Bochner \cite{Bochner_1942} gave an alternative proof, based on generalized Fourier analysis and uniform approximation. It also appears as Theorem 30.2 in \cite{Hewitt_1970} with Bochner's proof and more context.  We prove here that multiplicative linear functionals on $\A$ are automatically continuous.

The key tool from the compact case, namely the Haar measure, is not available here. Instead, we rely on dynamical properties of Roelcke-precompact groups, in the form of Lemma \ref{keylem}.

\begin{theo} \label{auto}
Let $G$ a Roelcke-precompact non-archimedean Polish group and let $\phi \colon \A \longrightarrow \C$ be multiplicative and linear. Then $\phi$ is bounded. 
\end{theo}

\begin{proof}

    We will show the positivity of $\phi$. To that aim, fix $f \in \A$. The exist $s_1,...,s_n\in \FP(\M_G)$ and $\lambda_1,...,\lambda_n \in \C$ such that $f= \sum_{i\leqslant n}{\lambda_i e_{s_i}}$. Recalling the identification $G = \Aut(\M_G)$ from Remark \ref{coset}, we set up for the application of Lemma \ref{keylem}. 
    
    Let $A = \bigcup\{\dom(s_i),\ i\leqslant n,\ \dom(s_i)\subseteq \dom(u_\phi)\}$ which is finite and let $I = \{i \leqslant n, \ \dom(s_i) \nsubseteq ~\acl(A)\}$.
    By Lemma \ref{keylem}, there exists $g_0 \in G$ that extends $u_{\phi|A}$ but none of the $s_i$ for $i\in I$. Exactly as in the proof of Corollary \ref{spa}, $g_0$ in fact satisfies:
    \begin{equation}\label{choix}
    \forall i\leqslant n, \ \left[ s_i\subseteq g_0 \iff s_i\subseteq u_\phi \right].
    \end{equation}

    It follows from (\ref{choix}) and the properties of $u_\phi$ that:
    $$\phi(f) = \sum_{i\leqslant n}{\lambda_i \phi(e_{s_i})} = \sum_{s_i\subseteq u_\phi}{\lambda_i} = \sum_{s_i\subseteq g_0}{\lambda_i} = f(g_0).$$
    In particular, if $f$ is such that $\forall g \in G, \ f(g) \geqslant 0$, then $\phi(f) \geqslant 0$. Hence $\phi$ is positive in a strong sense which is well-known to imply continuity (simply note that for every $f \in \A, \ ||f||_\infty\cdot 1 \pm f$ is in $\A$ and only takes positives values).

\end{proof}
 A direct consequence of this result is that the map from Corollary \ref{spa} lifts to an homeomorphism $\Hc(G) \longrightarrow \Pa(\M_G)$.

The Hilbert compactification comes with a semi-topological monoid structure. The composition law is a form of convolution. Sometimes seen as a variation of the \textit{Arens product}, it also appears in \cite{Bochner_1942}. It is built as follows:

For $\phi \in \Hc(G)$, one can define a map $\underline{\phi} \colon \A \longrightarrow \A$ by setting
    $$\forall f \in \A, \forall g \in G, \ \underline{\phi}(f)(g) = \phi(g^{-1} \cdot f)$$
To see that it is well defined, note that for $g \in G$ and $s \in \FP(\M_G)$, we have $\phi(g^{-1}\cdot e_s) = \phi(e_{g^{-1} \circ s}) = 1$ if $u_\phi$ extends $g^{-1} \circ s $ and $0$ otherwise. In other terms, 
    \begin{equation}\label{convocalc}
        \underline{\phi}(e_s) =  \left\{
    \begin{array}{ll}
        e_{s\circ (u_\phi^{-1})} &\text{ if } \dom(s) \subseteq \dom(u_\phi), \\
        0  &\mbox{otherwise.}
    \end{array}
\right.
    \end{equation}
Thus, in both cases, $\underline{\phi}(e_s)$ is an element of $\A$. Next, given $\phi$ and $\psi$ in $\Hc(G)$, we can define a map $\phi*\psi\colon \A \longrightarrow \C$ by setting $ \phi*\psi = \phi\circ \underline{\psi}$.

Then, given $s \in \FP(\M_G)$, using (\ref{convocalc}) we obtain: 
$$\phi*\psi(e_s) = \phi\circ \underline{\psi}(e_s)) = \left\{
    \begin{array}{ll}
        1 & \mbox{if } s \subseteq u_\phi \circ u_\psi, \\
        0 & \mbox{otherwise.}
    \end{array}
\right.$$ 
By Corollary \ref{spa}, $\phi*\psi$ is an element of $\spa$ which, by Theorem \ref{auto}, uniquely extends to an element of $\Hc(G)$. Moreover, we obtained along the way that the map $\phi \longmapsto u_\phi$ sends convolution in $\Hc(G)$ to composition where defined in $\Pa(\M_G)$. 

The Hilbert compactification also has an involution $\phi \longmapsto \phi^*$, which is given by:
$$\forall f\in \Hilb(G), \ \phi^*(f) = \overline{\phi(f^*)},$$
where $f^*(g) = \overline{f(g^{-1})}$ for every $f \in \Hilb(G)$ and $g \in G$. It also preserved by the map $\Pa(\M_G) \longrightarrow~\Hc(G)$. Indeed, if $s\in \FP(\M_G)$, we have:
$$\phi^*(e_s) = \overline{\phi(e_s^*)} = \phi(e_{s^{*}}) = \left\{
    \begin{array}{ll}
        1 & \mbox{if } s \subseteq u_\phi^*, \\
        0 & \mbox{otherwise,}
    \end{array}
\right.$$
hence $u_{\phi^*} =u_{\phi}^*$.

Consequently, we can identify the Hilbert compactification of $G$ with $\Pa(\M_G)$. We have in fact just proved the following (a similar statement obtained with a different approach appears as Theorem 0.3 in \cite{BenYaacov_2018}):

\begin{theo}\label{id1}
    The map $\Hc(G) \longrightarrow \Pa(\M_G)$ is an isomorphism of semi-topological monoidal compactifications of $G$.   
\end{theo}

\section{The Tannaka monoid: Another description of \texorpdfstring{$\Hc(G)$}{H(G)}}\label{Tannaka}

In \cite{Tannaka}, Tannaka associated to a compact group $K$ a monoid $\T(K)$ of \textit{operations} on the class of representations of $K$. More explicitly, an element of $\T(K)$ is a family of operators $(u_\pi)_\pi$ where $\pi$ ranges over all the finite dimensional representations of $K$ and $u_\pi$ is an operator on same the Hilbert space as $\pi$. Moreover, the family must commute with intertwining operators and preserve the common operations on representations: sum, tensor product and conjugation. The details can be found in \cite{Chevalley_1999} or \cite{Hewitt_1970}. A more recent treatment of this approach using efficient terminology from category theory appears in \cite{Joyal_1991}. Inspired by this take on Tannaka's duality, we carry a similar construction for Roelcke-precompact non-archimedean Polish groups. While Tannaka showed in the compact case that $\T(K)$ is a compact group canonically isomorphic to $K$, we will obtain a *-monoid canonically isomorphic to the Hilbert compactification.

Given a topological group $G$, we will denote by $\Rep(G)$ the category whose objects are representations of $G$ that split as a finite sum of irreducibles and whose morphisms are \textit{intertwining operators}. $\B'(G)$ will denote the algebra of matrix coefficients of $G$ arising from representations in $\Rep(G)$. Recall that $\widehat{G}$ denotes the unitary dual of $G$, i.e. the set of equivalence classes of irreducible unitary representations of $G$.

\begin{prop}\label{preli}
    Let $G$ be a Roelcke-precompact non-archimedean Polish group. The following properties hold:
    
    \begin{enumerate}
        \item $\Rep(G)$ is stable under tensor product of representations.
        \item Only countably many equivalence classes of representations appear in $\Rep(G)$.
        \item $\A \subseteq \B'(G)$.
    \end{enumerate}
\end{prop}

\begin{proof}
\; 

\begin{enumerate}
    \item Follows directly from Lemma \ref{tensor}.

    \item By definition of $\Rep(G)$ it suffices to see that $\widehat{G}$ is countable. For every $\lambda \in \widehat{G}$, fix a representative $\pi_\lambda$. We deduce from Fact \ref{oligorep} that:
    $$\bigoplus_{\lambda\in\widehat{G}} \pi_\lambda \leqslant \Lambda_G.$$

    Since $\Lambda_G$ is defined on $\ell^2(\M_G)$ where $\M_G$ is countable, $\widehat{G}$ must be countable.

    \item We have seen at the end of the proof of Lemma \ref{tensor} that $\ell^2(G/U)$ splits as a finite sum of irreducible subrepresentations for every open subgroup $U$ of $G$. Since $\A$ is generated by matrix coefficient arising from such representations, this proves the claim.
\end{enumerate}

\end{proof}

Let $\U_G$ be the forgetful functor from $\Rep(G)$ to the category of Hilbert spaces and let $\Nat(\U_G)$ denote the class of natural transformations of $\U_G$. Explicitly, an element of $\Nat(\U_G)$ is a family $u=(u_\pi)_{\pi\in \Rep(G)}$ where for every representation $\pi \in \Rep(G)$, $u_\pi$ is a bounded operator on $\Hi_\pi$ with the following condition. For every pair of representations $\pi_1,\pi_2\in \Rep(G)$ and every intertwining operator $h\in \Hom(\pi_1,\pi_2)$, the following diagram is commutative:

\begin{center}
\begin{tikzcd}[column sep=large, row sep=large]
    \Hi_{\pi_1} \arrow[r, "u_{\pi_1}"] \arrow[d, "h"'] &\Hi_{\pi_1} \arrow[d, "h"]\\
    \Hi_{\pi_2} \arrow[r, "u_{\pi_2}"] & \Hi_{\pi_2}
\end{tikzcd}
\end{center}

It is easily seen that $\Nat(\U_G)$ is a complex algebra under coordinatewise sum and composition. Moreover, elements of $\Nat(\U_G)$ behave well with subrepresentations:

\begin{lem}\label{stable}
    Let $G$ be a Roelcke-precompact non-archimedean Polish group and let $u\in \Nat(\U_G)$.

    \begin{enumerate}
        \item Let $\pi \in \Rep(G)$ and let $F$ be a closed subspace of $\Hi_\pi$ that is stable under the action of $G$. Then $F$ is also stable under $u_\pi$.

        \item Let $\pi_1, \pi_2 \in \Rep(G)$. Then:
        $$u_{\pi_1 \oplus \pi_2} = u_{\pi_1} \oplus u_{\pi_2}.$$
    \end{enumerate}
\end{lem}

\begin{proof}
\;

\begin{enumerate}
    \item Let $p_F \colon \Hi_\pi \longrightarrow \Hi_\pi$ be the orthogonal projection on $F$. Then $F$ is stable under the action of $G$ if and only if $p_F$ commutes with the action of $G$. In particular, if $F$ is $G$-stable then $u_\pi$ must also commute with $p_F$, i.e. $u_\pi(F) \subseteq F$.

    \item Let $\pi = \pi_1 \oplus \pi_2$. By definition, $\Hi_{\pi} = \Hi_{\pi_1} \oplus \Hi_{\pi_2}$. Moreover, the inclusion maps $\Hi_{\pi_1} \longrightarrow \Hi_\pi$ are intertwining operators. Thus, the following diagram must be commutative, which proves the claim:
    \begin{center}
    \begin{tikzcd}[column sep=large, row sep=large]
        \Hi_{\pi_1} \arrow[d] \arrow[r,"u_{\pi_1}"] &\Hi_{\pi_1} \arrow[d]\\
        \Hi_\pi \arrow[r, "u_\pi"]               &\Hi_\pi\\
        \Hi_{\pi_2} \arrow[u] \arrow[r, "u_{\pi_2}"] & \Hi_{\pi_2} \arrow[u]
    \end{tikzcd}
    \end{center}

\end{enumerate}
    
\end{proof}

The following is essentially Proposition 4 in \cite{Joyal_1991}. The proof is the same but we reproduce it for completeness. If $\Hi$ is Hilbert space, $\Li(\Hi)$ will denote the space of bounded operators $\Hi \longrightarrow \Hi$ and $\Li_1(\Hi)$ the subset of those operators with norm at most $1$.

\begin{prop}
    Let $G$ be a Roelcke-precompact non-archimedean Polish group. For every $\lambda \in \widehat{G}$, fix $\pi_\lambda \colon G \to \Un(\Ki_\lambda)$ a representative of $\lambda$. Then, the restriction map

    \begin{equation*}
    q\colon  \Nat(\U_G) \longrightarrow \prod_{\lambda \in \widehat{G}} \Li(\Ki_\lambda), \qquad
    u  \longmapsto (u_{\pi_\lambda})_{\lambda \in \widehat{G}}
    \end{equation*}
is a bijective correspondence that respects the algebra structures. In particular, $\Nat(\U)$ is a set.
    
\end{prop}

\begin{proof}
Lemma \ref{stable} together with the fact that our representations are finite sums of irreducibles show that an element $u$ is determined by its image under $q$. In particular, $q$ is injective. For surjectivity, fix $(t_\lambda) \in \prod_{\lambda \in \widehat{G}} \B(\Ki_\lambda)$. We will define a pre-image $u$ for $(t_\lambda)$. 

Let $\pi \colon G \to \Un(\Hi)$ be an element of $\Rep(G)$. There is a unique \textit{isotypical} decomposition $\Hi = ~\bigoplus_{\lambda} \Hi_\lambda$ into closed invariant subspaces where for every $\lambda \in \widehat{G}$, any irreducible subrepresentation of $\Hi_\lambda$ is isomorphic to $\pi_\lambda$. Since $\pi \in \Rep(G)$, the space $\Hom_G(\Ki_\lambda, \Hi_\lambda)$ is finite dimensional (this is an application of Schur's Lemma). Thus, there is a canonical $G$-equivariant isomorphism:
$$\Psi_\lambda\colon  \Ki_\lambda \otimes \Hom_G(\Ki_\lambda, \Hi_\lambda) \longrightarrow \Hi_\lambda.$$

Using $\Psi_\lambda$, we can define $u_\pi$ as follows:
$$u_\pi = \bigoplus_{\lambda} \left[ \Psi_\lambda \circ (t_\lambda \otimes 1 ) \circ \Psi_\lambda ^{-1}\right]$$

To see that it defines a natural transformation of $\U_G$, fix $\pi, \pi'\in \Rep(G)$ and write $\Hi = \Hi_\pi$ and $\Hi'= \Hi_{\pi'}$. Let $h\in \Hom(\pi,\pi')$. Recalling the classical general fact that $\Hom(\pi_\lambda,\pi_\mu) = 0$ if $\lambda \neq \mu \in \widehat{G}$, we see that $h$ preserves the isotypical decompositions of $\pi$ and $\pi'$:
$$\forall \lambda \in \widehat{G}, \ h(\Hi_\lambda) \subseteq \Hi'_\lambda$$
with the above notations. Moreover, these decompositions are also preserved by $u_\pi$ and $u_{\pi'}$ by construction. Hence, we can reduce to the case $\Hi = \Hi_\lambda$ and $\Hi' = \Hi'_\lambda$ for a single $\lambda \in \widehat{G}$, which leaves us with the following diagram:

    \begin{center}
    \begin{tikzcd}[column sep=large, row sep=large]
       \Hi_\lambda \arrow[d,"h"] \arrow[r,"\Psi_\lambda^{-1}"] 
            & \Ki_\lambda \! \otimes \! \Hom_G(\Ki_\lambda,\Hi_\lambda) \arrow[r,"t_\lambda \otimes 1"] \arrow[d,"1\otimes (h\circ \cdot)",shift right = 11]
                & \Ki_\lambda \hspace{-1mm} \otimes \hspace{-0,6mm} \Hom_G(\Ki_\lambda,\Hi_\lambda) \arrow[r,"\Psi_\lambda"] \arrow[d,"1\otimes (h\circ \cdot)",shift right = 11]
                    & \Hi_\lambda \arrow[d,"h"]\\
       \Hi'_\lambda  \arrow[r,"\Psi_\lambda^{'-1}"] 
            & \Ki_\lambda \! \otimes \! \Hom_G(\Ki_\lambda,\Hi'_\lambda)  \arrow[r,"t_\lambda\otimes 1"] 
                & \Ki_\lambda \! \otimes\! 
                \Hom_G(\Ki_\lambda,\Hi'_\lambda)  \arrow[r,"\Psi'_\lambda"]
                    & \Hi'_\lambda
    \end{tikzcd}
    \end{center}

    To conclude, we need to prove that the greatest square is commutative. But the three smallest squares are easily seen to be commutative, which is enough.
\end{proof}

Consider the weak operator topology on $\Li(\Hi)$ for every representation $G$. Then we can endow $\prod_{\lambda \in \widehat{G}} \Li(\Hi_\lambda)$ with the product topology. $\Nat(\U_G)$ can also be endowed with the coarsest topology for which the projection maps $u\longmapsto u_\pi$ are continuous for every representation $\pi$ of $G$. Then $q$ is also a homeomorphism for those topologies.

Tensor product is the last ingredient we need to add in order to define our version of the Tannaka monoid:

\begin{defi}\label{Tmonoid}
    The \textit{Tannaka monoid} $\T(G)$ of a Roelcke-precompact non-archimedean Polish group $G$ is the set of elements $u\in \Nat(\U_G)$ that are non-zero and commute with tensor product, i.e. such that:
    $$u_\mathbf{1}=\operatorname{Id}_\C \quad \text{and} \quad \forall \pi,\pi'\in \Rep(G),\ u_{\pi \otimes \pi'} = u_\pi \otimes u_{\pi'}$$
    where $\mathbf{1}$ denotes the trivial representation of $G$ on $\C$. The operation is coordinatewise composition: for every $u,v\in \T(G),\ (u\circ v)_\pi = u_\pi \circ v_\pi$. It also admits an involution $u\longmapsto u^*$ given by coordinatewise adjunction. It is endowed with the induced topology as a subspace of $\Nat(\U_G)$.
\end{defi}

Note that every $g$ in $G$ naturally defines an element $u^g\in \T(G)$. In the compact case, Tannaka essentially showed that this map is a group homeomorphism. In our context, we obtain the Hilbert compactification of $G$ again:

\begin{theo}
    Let $G$ be a Roelcke-precompact non-archimedean Polish group. Then $\T(G)$ is a compact semi-topological *-monoid isomorphic to the Hilbert compactification of $G$.
\end{theo}

\begin{proof}
    First, it is easily seen that composition is separately continuous in $\T(G)$ as it is in $\Li(\Hi)$ with the weak operator topology for every Hilbert space $\Hi$. Similarly, the involution $*$ of $\T(G)$ is also continuous.
    
    Next, we will define a map $\Phi \colon \T(G) \longrightarrow \Hc(G)$. Fix $u \in \T(G)$ in order to define a continuous linear multiplicative functional $\phi \colon  \Hilb(G) \longrightarrow \C$. For every $\pi\in \Rep(G)$ and $x,y\in \Hi_\pi$, denote by $f_{x,y}^\pi \colon g \longmapsto \left<\pi(g)x,y\right>$ the associated matrix coefficient. We first define $\phi$ on $\B'(G)$ by setting:
    \begin{equation}\label{dephi}
        \phi(f_{x,y}^\pi) = \left<u_\pi(x),y\right>.
    \end{equation}

    To see that it correctly defines a linear map $\B'(G) \longrightarrow \C$, fix $\pi_1,..., \pi_n \in \Rep(G)$ with underlying space $\Hi_1,...,\Hi_n$ respectively and let $x_1,y_1 \in \Hi_1,...,x_n,y_n\in \Hi_n$ and $\lambda_1,...,\lambda_n\in \C$ such that:
    \begin{equation}\label{link}
    \lambda_1f_{x_1,y_1}^{\pi_1} + ... + \lambda_nf_{x_n,y_n}^{\pi_n} = 0.
    \end{equation}
    Then, let $\pi = \oplus_{i\leqslant n} \pi_i, \  x= x_1+...+x_n$ and $y = \lambda_1y_1+ ...+ \lambda_n y_n$. Equation (\ref{link}) becomes: 
    $$\forall g \in G, \ 0= \left<\pi(g)x,y\right> = \left<x,\pi(g^{-1})y\right>,$$
    i.e. $x$ lies in $(\pi(G)y)^\perp$, a closed subspace of $\Hi$ stable under the action of $G$. Using Lemma \ref{stable}, we also have $u_{\pi}(x) \in (\pi(G)y)^\perp$ and even:
    \begin{equation*}    
        0 = \left<u_\pi(x),y\right>_\Hi =\lambda_1  \left<u_{\pi_1}(x_1),y_1\right>_{\Hi_1} + ... + \lambda_n \left<u_{\pi_n}(x_n),y_n\right>_{\Hi_n}.
    \end{equation*}
    This shows that $\phi$ is well defined on $\B'(G)$. It clearly is non-zero since $u$ is non-zero.
    
    It only remains to see that $\phi$ is multiplicative (since continuity is then automatic Proposition \ref{preli} and Theorem \ref{auto}). To that aim, fix $f_{x,y}^\pi, f_{x',y'}^{\pi'} \in \B'(G)$. We have:
    \begin{align*}
        \phi(f_{x,y}^\pi)\phi(f_{x',y'}^{\pi'}) 
        &= \left<u_\pi(x),y\right>\left<u_{\pi'}(x'),y'\right>\\
        &= \left<u_\pi(x)\otimes u_{\pi'}(x'), y\otimes y'\right>\\
        &= \left<u_{\pi}\otimes u_{\pi'}(x\otimes x'),y\otimes y'\right>\\
        &= \left<u_{\pi\otimes \pi'}(x\otimes x'),y\otimes y'\right>\\
        &=\phi(f_{x\otimes x', y\otimes y'}^{\pi\otimes \pi'})\\
        &=\phi(f_{x,y}^{\pi}f_{x',y'}^{\pi'}),
    \end{align*}
where we used the fact that $u \in \T(G)$ to obtain the fourth equality. Multiplicativity is proved hence $\Phi \colon \T(G) \longrightarrow \Hc(G)$ is well defined. The identity (\ref{dephi}) above also shows that this map is continuous as $\Hc(G)$ is endowed with the pointwise convergence (=weak*) topology.

To define the reciprocal $\Psi \colon  \Hc(G) \longrightarrow \T(G)$, fix $\phi \in \Hc(G)$ and $\pi \in \Rep(G)$. Then, the map
$$\Hi_\pi\times \Hi_\pi \longrightarrow \C, \qquad (x,y) \longmapsto \phi(f_{x,y}^\pi)$$
is a sesquilinear form. It is also bounded since for every $x,y\in \Hi$:
$$|\phi(f_{x,y}^\pi)|\leqslant ||\phi||\cdot ||f_{x,y}^\pi||_\infty \leqslant ||\phi||\cdot ||x||\cdot ||y||.$$
Using the Riesz representation theorem, there exists a unique bounded operator $u_\pi \in \Li(\Hi_\pi)$ such that:
\begin{equation}\label{cont}
    \forall x,y\in \Hi_\pi, \ \phi(f_{x,y}^\pi) = \left<u_\pi(x),y\right>.
\end{equation}
It is straightforward to show that $(u_\pi)_{\pi\in \Rep(G)}$ lies in $\T(G)$ and that (\ref{cont}) implies the continuity of the map $\Psi \colon \Hc(G) \longrightarrow \T(G)$. Equations (\ref{dephi}) and (\ref{cont}) together imply that those maps are inverse to each other, hence homeomorphisms.

To finish the proof, it only remains to show that the monoid structures are preserved. To see this, fix $u,v\in \T(G)$, and let $\pi \in \Rep(G)$. For every $x,y\in \Hi_\pi$ the following holds:
    \begin{align*}
        \Phi(u\circ v)(f_{x,y}^\pi) 
        &= \left<u_\pi\circ v_\pi(x),y\right>\\
        &=\Phi(u)(f_{v_\pi(x),y}^\pi)\\
        &= \Phi(u)\left[g\mapsto \left<v_\pi(x), \pi(g^{-1})y\right>\right]\\
        &=\Phi(u)\left[g\mapsto \Phi(v)(f_{x,\pi(g^{-1})y}^\pi)\right]\\
        &= \Phi(u) \left[g \mapsto \Phi(v)(g^{-1}\cdot f_{x,y}^\pi)\right]\\
        &= \Phi(u)* \Phi(v)(f_{x,y}^\pi).
    \end{align*}
\end{proof}

\section{Additional properties of \texorpdfstring{$\Hc(G)$}{H(G)}}\label{Conc}

We end this article by recovering previously known basic properties of $\Hc(G)$ and by turning the previous result into a proper duality, i.e a contra-variant functor that is full and faithful. 

First, recall that there is a natural map $\iota \colon G \longrightarrow \Hc(G)$ that sends $g\in G$ to the \textit{evaluation map at $g$}. Viewing $\Hc(G)$ as $\T(G)$, $\iota(g)$ is given by:
$$\forall \pi \in \Rep(G),\ \iota(g)_\pi = \pi(g)$$
Finally, viewing $\Hc(G)$ as $\Pa(\M_G)$ and identifying $G$ with $\Aut(\M_G)$ by Remark \ref{coset}, $\iota$ becomes the canonical inclusion $G \hookrightarrow \Pa(\M_G)$. It allows for an abstract reconstruction of $G$ from $\Hc(G)$:

\begin{theo}\label{recons}
    Let $G$ be a non-archimedean Roelcke-precompact Polish group. The natural map $$\iota \colon G \longrightarrow \Hc(G)$$ is a homeomorphic embedding such that:
    $$\forall g,h\in G, \ \iota(gh) = \iota(g) * \iota(h) \quad \& \quad \iota (g^{-1}) = \iota(g)^*.$$
    Moreover, $\iota(G)$ is exactly the set of invertible elements of $\Hc(G)$. In other words, $G$ is canonically isomorphic, as a topological group, to the set of invertible elements of $\Hc(G)$.
\end{theo}

\begin{proof}
    View $\Hc(G)$ as $\Pa(\M_G)$. The restriction of the topology of $\Pa(\M_G)$ to  $\Aut(\M_G)$ is easily seen to coincide with the pointwise convergence topology, since $\Aut(\M_G)$ does not see the open sets of the second kind in Remark \ref{topo}. Since we identified $G$ with $\Aut(\M_G)$ as topological groups in Remark \ref{coset}, our map is indeed a homeomorphic embedding. The other properties are obviously satisfied by the inclusion map $G \hookrightarrow \Pa(\M_G)$.
\end{proof}

Finally, we can reformulate the previous results as an equivalence of categories, or more precisely a duality. Let $\rpna$ be the category whose objects are the Roelcke-precompact non-archimedean Polish groups and arrows the continuous group morphisms. Note that if $p\colon G \longrightarrow H$ is a continuous group morphism between topological groups, it induces a functor $\widetilde{p} \colon \Rep(H) \longrightarrow \Rep(G)$ by setting $\widetilde{p}(\pi) = \pi\circ p$ for every $\pi \in \Rep(H)$ and $\widetilde{p}(h) = h$ for every morphism $T$ between representations of $H$. The functor $\widetilde{p}$ has additional properties, it is \textit{admissible} in the following sense:

\begin{defi}\label{admissible}
    Let $G,H$ be topological groups. A functor $F \colon  \Rep(H) \longrightarrow \Rep(G)$ is said to be \textit{admissible} if it satisfies the following properties: 

\begin{itemize}
    \item[(i)] $F$ sends the trivial one dimensional representation of $H$ to the trivial one dimensional representation of $G$.

    \item[(ii)] For every $\pi \in \Rep(H)$, $F(\pi)$ has the same underlying Hilbert space as $\pi$.
    
    \item[(iii)] For every $\pi,\pi' \in \Rep(H)$ and $h \in \Hom(\pi,\pi')$, $F(h) = h$.

    \item[(iv)] $F$ commutes with tensor product of representations, i.e. for every $\pi, \pi' \in \Rep(G)$,
    $$F(\pi \otimes \pi') = F(\pi) \otimes F(\pi').$$
\end{itemize}

\end{defi}

The collection of such functors is stable under composition and contains the identity maps. We can thus form the category $\CR$ whose class of objects is $(\Rep(G))_{G\in \rpna}$ and whose arrows are the admissible functors. We obtain the following, which is also true for compact (resp. locally compact abelian) groups by the usual duality theories of Pontryagin--van Kampen and Tannaka-Krein.

\begin{theo}
    The canonical contravariant functor 
    \begin{align*}
    \Rep\colon \rpna &\longrightarrow \CR\\
    G &\longmapsto \Rep(G) \\
    p &\longmapsto \widetilde{p}
    \end{align*} 
is a duality.
\end{theo}

\begin{proof}
It is clear by construction that $\Rep$ is surjective with regard to objects. To prove that it is faithful, let $G,H \in~ \rpna$ and assume that $p,q \colon G \longrightarrow H$ are such that $\widetilde{p} = \widetilde{q}$. Then, for every $\pi \in \Rep(G)$, we have $\pi\circ p = \widetilde{p}(\pi) = \widetilde{q}(\pi) = \pi\circ q $. Since Roelcke-precompact Polish groups satisfy the Gel'fand-Raikov Theorem (i.e. irreducible representations separate points \cite[Th.1.1]{Tsankov_2012}), it implies that $p=q$.

Finally, we prove that $\Rep$ is full. Let $G,H\in \rpna$ and let $F \colon \Rep(H) \longrightarrow \Rep(G)$ be an admissible functor. Let us prove that there exists $p\colon G \longrightarrow H$ such that $F = \widetilde{p}$. First, we show that $F$ induces a map $p_F\colon \T(G) \longrightarrow \T(H)$. Indeed, let $u\in \T(G)$ and $\pi \in \Rep(H)$. Since $F$ is admissible, we can define $p_F(u)$ at $\pi$ by setting $p_F(u)_{\pi} = u_{F(\pi)}$. Now, let $\pi, \pi' \in \Rep(H)$ and let $h\in \Hom(\pi,\pi')$. Since $F$ is a functor, $F(h)$ lies in $\Hom(F(\pi), F(\pi'))$. Hence, by definition of $\T(G)$, the following diagram is commutative:

\begin{center}
\begin{tikzcd}[column sep=large, row sep=large]
    \Hi_{F(\pi)} \arrow[r, "u_{F(\pi)}"] \arrow[d, "F(h)"'] &\Hi_{F(\pi)} \arrow[d, "F(h)"]\\
    \Hi_{F(\pi')} \arrow[r, "u_{F(\pi')}"] & \Hi_{F(\pi')}
\end{tikzcd}
\end{center}
Since $F$ is admissible, the above diagram is \textit{equal} to the following:

\begin{center}
\begin{tikzcd}[column sep=large, row sep=large]
    \Hi_{\pi} \arrow[r, "p_F(u)_{\pi}"] \arrow[d, "h"'] &\Hi_{\pi} \arrow[d, "h"]\\
    \Hi_{\pi'} \arrow[r, "p_F(u)_{\pi'}"] & \Hi_{\pi'}
\end{tikzcd}
\end{center}
Thus, in order to prove to $p_F(u)$ is an element of $\T(H)$, it only remains to see that $p_F(u)$ respects tensor products. By admissibility of $F$: 
$$p_F(u)_\triv = u_{F(\triv)} = u_\triv = \operatorname{Id}_\C,$$
and, for every $\pi,\pi' \in \Rep(H)$, using the fact that $u$ is in $\T(G)$,
$$p_F(u)_{\pi\otimes \pi'} = u_{F(\pi\otimes \pi')} = u_{F(\pi)\otimes F(\pi')} = u_{F(\pi)} \otimes u_{F(\pi')} = p_F(u)_\pi \otimes p_F(u)_{\pi'}.$$
Hence $p_F(u)$ is an element of $\T(H)$. Clearly, $p_F$ is continuous and respects the monoid structures. Recalling the last statement of Theorem \ref{recons}, we see that $p_F$ restricts to a continuous group morphism $p \colon G \longrightarrow H$ that induces $F$.
\end{proof}

\bibliographystyle{amsalpha}
\bibliography{bib.bib}

\end{document}